\documentclass[a4paper,11pt]{article}
\usepackage{amsthm,amsmath,amssymb}
\usepackage[OT1]{fontenc}
\usepackage{graphicx}
\usepackage[utf8]{inputenc}
\usepackage{enumitem}

\usepackage{tabularx,booktabs}

\usepackage{authblk}

\DeclareGraphicsExtensions{.pdf}

\usepackage[usenames,dvipsnames]{xcolor}

\newtheorem{theorem}{Theorem}
\newtheorem{lemma}{Lemma}

\numberwithin{lemma}{section}

\newtheorem{question}{Question}
\newtheorem*{claim}{Claim}

\newtheorem{reduction}{Reduction Lemma}

\newenvironment{claimproof}{%
\noindent%
\textit{Proof of Claim.}%
}
{
\hfill $\triangle$%
\medskip
}

\let\leq\leqslant
\let\geq\geqslant
\let\setminus\smallsetminus

\newcommand{\floor}[1]{{\left\lfloor #1 \right\rfloor}}

\makeatletter
\let\old@setaddresses\@setaddresses
\def\@setaddresses{\bgroup\parindent 0pt\let\scshape\relax\old@setaddresses\egroup}
\makeatother

\begin{document}

\title{Chromatic number of ordered graphs with forbidden ordered subgraphs.}

\author{Maria Axenovich}
\author{Jonathan Rollin}
\author{Torsten Ueckerdt}
\affil{Department of Mathematics, Karlsruhe Institute of Technology} 

\maketitle
 
\begin{abstract}
 It is well-known that the graphs not containing a given graph $H$ as a subgraph have bounded chromatic number if and only if $H$ is acyclic.
 Here we consider \emph{ordered graphs}, i.e., graphs with a linear ordering $\prec$ on their vertex set, and the 
 function
 \[
  f_\prec(H) = \sup\{\chi(G)\mid G\in {\rm Forb}_\prec(H)\}, 
 \]
where ${\rm Forb}_\prec(H)$ denotes the set of all ordered graphs that do not contain a copy of $H$.
 
%
%
 If $H$ contains a cycle, then as in the case of unordered graphs, $f_\prec(H) = \infty$. However, in contrast to the unordered graphs,   we describe an infinite family of ordered forests $H$ with $f_\prec(H) = \infty$.
  An ordered graph is crossing if there are two edges $uv$ and $u'v'$ with $u\prec u'\prec v\prec v'$.
 For connected crossing ordered graphs $H$ we reduce the problem of determining whether $f_\prec(H) \neq \infty$ to a family of so-called monotonically alternating trees.
For non-crossing $H$ we prove that $f_\prec(H) \neq \infty$ if and only if $H$ is acyclic and does not contain a copy of any of  the five special ordered forests on four or five vertices, which we call bonnets.
 For such forests $H$, we show that $f_\prec (H) \leq 2^{|V(H)|}$ and that $f_\prec (H) \leq 2|V(H)|-3$ if $H$ is connected. 
\end{abstract}

  
\section{Introduction}

What conclusions can one make about the chromatic number of a graph knowing that it does not contain certain subgraphs?
Let $H$ be a graph on at least two vertices, ${\rm Forb}(H)$ be the set of all graphs not containing $H$ as a subgraph, and 
 $f(H) = \sup\{\chi(G)\mid G\in {\rm Forb}(H)\}.$
If $H$ has a cycle of length $\ell$, then for any integer $\chi$ there is a graph $G$ of girth at least $\ell+1$ and chromatic number $\chi$, see  \cite{LargeGirthHighChromatic}, 
implying   that $f(H)=\infty$.
On the other hand, if $H$ is a forest  on $k$ vertices and $G$ is a graph of chromatic number at least $k$, then $G$ contains a $k$-critical subgraph $G'$, that in turn has minimum degree at least $k-1$. Thus a copy of $H$ can be found as a subgraph of $G'$ by a greedy embedding.
Therefore $G\not\in {\rm Forb}(H)$, implying that $f(H) \leq k-1$.
So, we see that $f(H)$ is finite if and only if $H$ is acyclic.

A similar situation holds for directed graphs, with a similarly defined function $f_{\rm dir}(H)$ being finite if and only if the underlying graph of $H$ is acyclic.
A result of Addalirio-Berry \textit{et al.}~\cite{A-B}, see also~\cite{BurrDirectedSubgraphs}, implies that $f_{{\rm dir}}(H)\leq k^2/2 - k/2 -1 $ whenever $H$ is a directed $k$-vertex graph whose underlying graph is acyclic.\\

Here, we consider the behavior of the chromatic number of ordered graphs with forbidden ordered subgraphs. 
An \emph{ordered graph} $G$ is a graph $(V,E)$ together with a linear ordering $\prec$ of its vertex set $V$.
An \emph{ordered subgraph} $H$ of an ordered graph $G$ is a subgraph of the (unordered) graph $(V,E)$ together with the linear ordering of its vertices inherited from $G$.
An ordered subgraph $H$ is a \emph{copy of an ordered graph} $H'$ if there is an order preserving isomorphism between $H$ and $H'$.
For an ordered graph $H$ on at least two vertices\footnote{If $H$ has only one vertex, then ${\rm Forb}_\prec(H)$ consists only of the graph with empty vertex set and one can think of $f_\prec(H)$ as being equal to $0$. However, we will avoid this pathologic case throughout.} let ${\rm Forb}_\prec(H)$ denote the set of all ordered graphs that do not contain a copy of $H$.
We consider the function $f_\prec$ given by
\[
 f_\prec(H) = \sup\{\chi(G)\mid G\in {\rm Forb}_\prec(H)\}.
\]
 
We show that it is no longer true that $f_\prec(H)$ is finite if and only if $H$ is acyclic.  
When $H$ is connected, we reduce the problem of determining whether $f_\prec(H) \neq \infty$ to a well behaved class of trees, which we call monotonically alternating trees.
We completely classify  so-called ``non-crossing" ordered graphs $H$ for which $f_\prec(H)  = \infty$.
In case of   ``non-crossing"  $H$ with finite $f_\prec(H)$, we provide specific upper bounds on this function in terms of the number of vertices in $H$.
Note that $f_\prec(H)\geq |V(H)|-1$ for any ordered graph $H$, since a complete graph on $|V(H)|-1$ vertices is in  ${\rm Forb}_\prec(H)$.

\medskip
We need some formal definitions before stating the main results of the paper.
We consider the vertices of an ordered graph laid out along a horizontal line according to their ordering $\prec$  and say that 
for $u\prec v$  the vertex $u$ is to the \emph{left of} $v$ and the vertex $v$ is to the \emph{right of} $u$. 
We write $u\preceq v$ if $u\prec v$ or $u=v$.
For two sets of vertices $U$ and $U'$ we write $U\prec U'$ if all vertices in $U$ are left of all vertices in $U'$.
Two edges $uv$ and $u'v'$  \textbf{\textit{cross}} if $u\prec u'\prec v\prec v'$ and an ordered graph $H$ is called \emph{crossing} if it contains two crossing edges.
Otherwise, $H$ is called \emph{non-crossing}.
Two distinct ordered graphs $G$ and $H$ \emph{cross each other} if there is an edge in $G$ crossing an edge in $H$.\\

An ordered graph is a \textbf{\textit{bonnet}} if it has $4$ or $5$ vertices $u_1\prec u_2\preceq u_3\prec u_4\preceq u_5$ and edges $u_1u_2, u_1 u_5, u_3u_4$, or if it has vertices $u_1 \preceq u_2\prec u_3 \preceq u_4 \prec u_5$ and edges $u_1u_5, u_4u_5, u_2u_3$.
See Figure~\ref{fig:BonnetTutte} (first two rows).  
An ordered path $P=u_1,\ldots,u_n$ is a \textbf{\textit{tangled path}} if for a vertex $u_i$, $1<i<n$, that is either leftmost or rightmost in $P$ there is an edge in the subpath $u_1,\ldots,u_i$ that crosses an edge in  the subpath $u_i,\ldots,u_n$.
See Figure~\ref{fig:BonnetTutte} (last row, left and middle).
Note that there are crossing paths which are not tangled, see for example Figure~\ref{fig:BonnetTutte} (right).

\begin{figure}
 \centering
 
 \hfill
 \begin{minipage}{0.30\textwidth}
  \centering
  \includegraphics{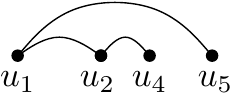}
 \end{minipage}
 \hfill
 \begin{minipage}{0.30\textwidth}
  \centering
  \includegraphics{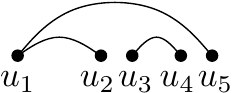}
 \end{minipage}
 \hfill
 \begin{minipage}{0.30\textwidth}
  \centering
  \includegraphics{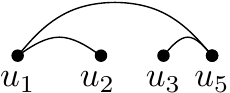}
 \end{minipage}
 \hfill
 \\[10pt]
 \hfill
 \begin{minipage}{0.30\textwidth}
  \centering
  \includegraphics{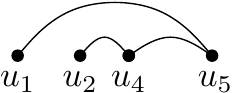}
 \end{minipage}
 \hfill
 \begin{minipage}{0.30\textwidth}
  \centering
  \includegraphics{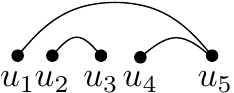}
 \end{minipage}
 \hfill
 \begin{minipage}{0.30\textwidth}
  \centering
  \includegraphics{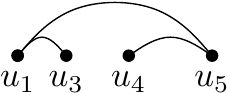}
 \end{minipage}
 \hfill
 \\[10pt]
 \hfill
 \begin{minipage}{0.30\textwidth}
  \centering
  \includegraphics{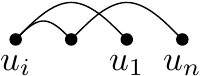}
 \end{minipage}
 \hfill
 \begin{minipage}{0.30\textwidth}
  \centering
  \includegraphics{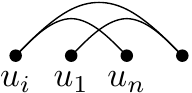}
 \end{minipage}
 \hfill
 \begin{minipage}{0.30\textwidth}
  \centering
  \includegraphics{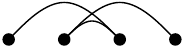}
 \end{minipage}
 \hfill

 \caption{All bonnets (first two rows), two tangled paths (last row, left and middle) and a crossing path that is not tangled (last row, right).}
 \label{fig:BonnetTutte}
\end{figure}


\begin{theorem} \label{Tutte-Shift}
 If an ordered graph $H$ contains a cycle, a bonnet, or a tangled path, then $f_\prec(H) = \infty$.
\end{theorem}

A vertex $v$ of an ordered graph $G$ is called \emph{inner cut vertex}, if there is no edge $uw$ with $u\prec v\prec w$ in $G$ and $v$ is not leftmost or rightmost in $G$.
An \emph{interval} in an ordered graph $G$ is a set $I$ of vertices such that for all vertices $u,v\in I$, $x\in V(G)$ with $u\prec x\prec v$ we have $x\in I$.
A  \textbf{\textit{segment}} of an ordered graph $G$ with $|V(G)|\geq 2$ is an induced subgraph $H$ of $G$ such that $|V(H)|\geq 2$, $V(H)$ is an interval in $G$, the leftmost and rightmost vertices in $H$ are either inner cut vertices of $G$ or leftmost respectively rightmost in $G$, and all other vertices in $H$ are not inner cut vertices in $G$.
So, $G$ is the union of its segments, any two segments share at most one vertex and the inner cut vertices of~$G$ are precisely the vertices contained in two segments of $G$.
In particular, the number of inner cut vertices of $G$ is exactly one less than the number of its segments. See Figure \ref{fig:segments}.

\begin{figure}
 \centering
 \includegraphics{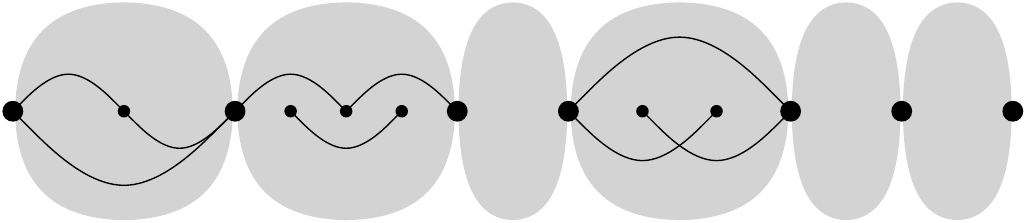}
 \caption{Segments of an ordered graph. The bold vertices are either inner cut-vertices or left-, rightmost vertices.}
 \label{fig:segments}
\end{figure}

The \emph{length} of an  edge $xy$ is the number of vertices $v$ such that $x\preceq v\prec y$.
A shortest edge among all the edges incident to a vertex $x$ is referred to as a \emph{shortest edge incident to} $x$. Note that there is either   $1$ or $2$  shortest edges incident to a given vertex in a 
connected graph on at least two vertices. 
Let $U$ be a vertex set in an ordered tree $T$, such that each vertex in $U$ has exactly one shortest edge incident to it. 
For such a set $U$, let  $S(U)$  be the set of edges $e_u$ such that   $e_u$ is a shortest edge incident to $u$, $u\in U$. 
  We call an ordered tree $T$ 
 \textbf{\textit{monotonically alternating}} if there is a partition $V(T)=L\dot\cup R$, with $L\prec R$, such that $L$ and $R$ are independent sets in $T$, $E=S(L)\cup S(R)$, and neither $S(L)$ nor $S(R)$ contains a pair of crossing edges.

\begin{figure}
 \centering
 \includegraphics{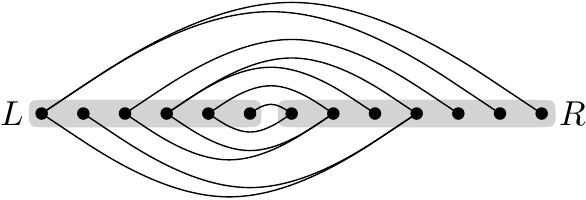}
 \caption{A monotonically alternating tree. Each edge on top is the shortest edge incident to a vertex in $R$ and each edge at the bottom is the shortest edge incident to a vertex in $L$.}
 \label{fig:monoAlt}
\end{figure}

\begin{theorem}\label{structural}
 An ordered tree $T$  contains neither a bonnet nor a tangled path if and only if each segment of $T$ is monotonically alternating.
 In particular if $f_\prec(H)\neq\infty$ for some connected ordered graph $H$, then each segment in $H$ is a monotonically alternating tree.
\end{theorem}

Recall that an ordered graph is non-crossing if it does not contain any crossing edges.
Note that a non-crossing graph does not contain tangled paths.

\begin{theorem}\label{non-crossing}
 Let $T$ be a non-crossing ordered graph on $k$ vertices.
 Then $f_\prec(T)\neq\infty$ if and only if $T$ is a forest that does not contain a bonnet.
 
 Moreover, if $f_\prec(T)\neq\infty$ then $k-1\leq f_\prec(T)\leq 2^k$.
 If, in addition $T$ is connected, then $f_\prec(T)\leq 2k-3$.
 Finally, for each $k\geq 4$ there is an ordered non-crossing tree $T$ with $k\leq f_\prec(T)\neq\infty$, while for $k=2,3$ we have $f_\prec(T) = k-1$.
\end{theorem}

 For certain classes of ordered forests we prove better upper bounds on $f_\prec$.
 A \emph{$k$-nesting} is an ordered graph $T$ on vertices $u_1\prec\cdots\prec u_k\prec v_k\prec\cdots\prec v_1$ and edges $u_iv_i$, $1\leq i\leq k$.
 A \emph{$k$-crossing} is an ordered graph $T$ on vertices $u_1\prec\cdots\prec u_k\prec v_1\prec\cdots\prec v_k$ and edges $u_iv_i$, $1\leq i\leq k$. 
 We may omit the parameter $k$ if it is not important. 
 A \emph{generalized star} is a union of a star and isolated vertices.

 The following theorem summarizes several results on trees which are either not covered by Theorem~\ref{non-crossing} or improve the upper bound from Theorem~\ref{non-crossing} significantly.
 
One of the known classes of such graphs is a special family of star forests, or, in other words, tuple matchings.
For positive integers $m$ and $t$ and a permutation $\pi$ of $[t]$, an {\it $m$-tuple $t$-matching}  $M=M(t,m, \pi)$ is an ordered graph with vertices $v_1\prec\cdots\prec v_{t(m+1)}$, where each edge is of the form $v_iv_{t+j+m(\pi(i)-1)}$ for $1\leq i\leq t$, $1\leq j\leq m$.
I.e., an $m$-tuple $t$-matching is a vertex disjoint union of $t$ stars on $m$ edges each, where $v_1,\ldots,v_t$ are the centers of the stars that are to the left of all leaves and the leaves of each star form an interval in $M$, so that these intervals are ordered according to the permutation $\pi$. The third item in the following theorem is an immediate corollary 
of a result by Weidert~\cite{Weidert} who provides a linear upper bound on the   the extremal function for $M$.
The other results are based on linear upper bounds for  the extremal functions of nestings due to Dujmovic and Wood~\cite{DW04}, on the extremal function of crossings due to Capoyleas and Pach~\cite{CapoyleasPach} and lower bounds for ordered Ramsey numbers due to Conlon \textit{et al.}~\cite{ConlonFoxLeeSudakov}, see also Balko et al. \cite{BalkoCibulkaKralKyncl}.
See Section~\ref{connections} for a more detailed description of extremal functions and ordered Ramsey numbers.

\begin{theorem}\label{others} 
 Let $T$ be an ordered forest on $k$ vertices.
 \begin{itemize}
  \item If each segment of $T$  is either a generalized star, a $2$-nesting, or a $2$-crossing, then $f_\prec(T) = k-1$.
  
  \item If each segment of $T$ is either a nesting, a crossing, a generalized star, or a non-crossing tree without bonnets, then $k-1\leq f_\prec(T)\leq  2k-3$.
  
  \item If $T$ is a tuple matching, then $k-1\leq f_\prec(T) \leq  2^{10k\log(k)}$.
  
  \item There is a positive constant $c$ such that for each even  positive integer $k\geq 4$ there is a matching $M$ on $k$ vertices with  $f_\prec(M) \geq 2^{c\frac{\log(k)^2}{\log\log(k)}}$.
 \end{itemize}
 \end{theorem}
 

The paper is organized as follows.
In Section~\ref{definitions} we introduce all missing necessary notions.
In Section~\ref{connections} we summarize the known results on extremal functions and Ramsey numbers for ordered graphs and show how they could be used in determining  $f_\prec$.
In Section~\ref{reductions} we prove  some structural lemmas and provide several reductions that are used in the proofs of the main results and that might be of independent interest.
Section~\ref{proofs} contains the proofs of Theorems~\ref{Tutte-Shift}--\ref{others}.
We summarize all known results for forests with at most three edges in Section~\ref{small-forests}. 
Finally, Section~\ref{conclusions} contains conclusions and open questions.


 \section{Definitions}\label{definitions}

Let $K_n$ denote a complete graph on $n$ vertices.
For a positive integer $n$ and an ordered graph $H$, let ${\rm ex}_\prec(n,H)$ denote the \emph{ordered extremal number}, i.e., the largest number of edges in an ordered graph on $n$ vertices in ${\rm Forb}_\prec(H)$.
For an ordered graph $H$ the \emph{ordered Ramsey number} $R_\prec(H)$ is the smallest integer $n$ such that in  any edge-coloring of an ordered $K_n$ in two colors there is a  monochromatic copy of $H$.
Recall that an interval in an ordered graph $G$ is a set $I$ of vertices such that for all vertices $u,v\in I$, $x\in V(G)$ with $u\prec x\prec v$ we have $x\in I$.
The \emph{interval chromatic number} $\chi_\prec(G)$ of an ordered graph $G$ is the smallest number of intervals, each inducing an independent set in $G$, needed to partition $V(G)$.
An inner cut vertex $v$ of an ordered graph $G$ \emph{splits $G$ into ordered graphs $G_1$ and $G_2$} if $G_1$ is induced by all vertices $u$ with $u\preceq v$ in $G$ and $G_2$ is induced by all vertices  $u$ with $v\preceq u$.
A vertex of degree $1$ is called a \emph{leaf}.
A vertex in an ordered graph $G$ is called \emph{reducible}, if it is a leaf in $G$, is leftmost or rightmost in $G$ and has a common neighbor with the vertex next to it.
We call an edge $uv$ in a graph $G$ \emph{isolated} if $u$ and $v$ are leaves in $G$.
A graph $G$ is \emph{$t$-degenerate} if each subgraph of $G$ has a vertex of degree at most $t$.
 A vertex $v$ is \emph{between} vertices $u$ and $w$ if $u\preceq v\preceq w$. 
The \emph{reverse} $\overline{G}$ of an ordered graph $G$ is the ordered graph obtained by reversing the ordering of the vertices in $G$.
A \emph{$u$-$v$-path} $P$ is a path starting with $u$ and ending with $v$, i.e., a path $v_1,\ldots, v_k$ with $u=v_1$, $v=v_k$.
Given a path $P=v_1,\ldots, v_k$ let $v_iP=v_i,\ldots, v_k$ and $Pv_i=v_1,\ldots, v_i$.
Similarly for a neighbor $v\not\in V(P)$ of $v_1$ let $vP =v,v_1,\ldots, v_k$.
If $U \subseteq V(G)$, $F\subseteq E(G)$ let $G[U]$, $G-U$ and $G-F$  denote the graphs $(U, E(G)\cap \binom{U}{2})$, $(V(G)\setminus U,E(G) \cap \binom{V(G)-U}{2})$, and $(V(G),E(G)\setminus F)$, respectively.
In particular if $u,v\in V(G)$ then $G-\{u,v\}$ is the graph obtained by removing $u$ and $v$ from $G$, not the edge $uv$ only.
If $u\in V(G)$ let $G-u=G-\{u\}$.
The definitions of tangled paths, bonnets, crossing edges and subgraphs, intervals, segments,  inner cut-vertices, and monotonically alternating trees are given before the statements of the main theorems in the introduction.
We shall typically denote a general ordered graph by $H$,  a tree or a forest by $T$, and a larger ordered graph by $G$.
For all other undefined graph theoretic notions we refer the reader to West~\cite{West}.


\section{Connections to known results}\label{connections}

There are connections between the extremal number ${\rm ex}_\prec(n,H)$ and the function $f_\prec(H)$. 
If there is a constant $c$ such that ${\rm ex}_\prec(n,H) < c\, n$ for every $n$, then 
\begin{equation}\label{f-ex}
 f_\prec(H) \leq 2c,
\end{equation} 
 so $f_\prec(H)$ is finite.
Indeed, if ${\rm ex}_\prec(n,H) < c\, n$ then any $G\in{\rm Forb}_\prec(H)$ has less than $c\,|V(G)|$ edges, and hence has a vertex of degree less than $2c$.
Thus if $G\in{\rm Forb}_\prec(H)$, then each subgraph of $G$ is in ${\rm Forb}_\prec(H)$, so each subgraph has a vertex of degree less than $2c$, so $G$ is $(2c-1)$-degenerate.
Therefore $\chi(G)\leq 2c$.

Ordered extremal numbers are studied in detail in~\cite {PachTardos}.
Recall that $\chi_\prec(G)$ is the smallest number of intervals, each inducing an independent set, needed to partition the vertices of an ordered graph $G$.
Pach and Tardos~\cite{PachTardos} prove that  for each ordered graph $H$
 \[{\rm ex}_\prec(n,H) = \left(1-\frac{1}{\chi_\prec(H)-1}\right)\binom{n}{2}+o(n^2).\]

For ordered graphs with interval chromatic number $2$, Pach and Tardos find a tight relation between the ordered extremal number and pattern avoiding matrices.
For an ordered graph $H$ with $\chi_\prec(H)=2$ let $A(H)$ denote the $0$-$1$-matrix where the rows correspond to the vertices in the first color and the columns to the vertices in the second color of a proper interval coloring of $H$ in $2$ colors and let $A(H)_{u,v}=1$ if and only if $uv$ is an edge in $H$.
A $0$-$1$-matrix $B$ \emph{avoids} another $0$-$1$-matrix $A$ if there is no submatrix in $B$ which becomes equal to $A$ after replacing some ones with zeros.
For a $0$-$1$-matrix $A$ let ${\rm ex}(n,A)$ denote the largest number of ones in an $n\times n$ matrix avoiding $A$.
In~\cite{PachTardos} it is shown that for each ordered graph $H$ with $\chi_\prec(H)=2$ there is a constant $c$ such that
 ${\rm ex}(\floor{\tfrac{n}{2}},A(H)) \leq {\rm ex}_\prec(n,H) \leq c\, {\rm ex}(n,A(H))\log n.$
 Thus, when ${\rm ex}(n,A(H))$ is linear in $n$, one can guarantee that ${\rm ex}_\prec (n, H) = O(n\log n)$, but this is not enough to claim that $f_\prec (H)\neq \infty$.

In addition, we see that there is no direct connection between  $f_\prec(H)$ and ${\rm ex}_\prec (n, H)$   because  there are dense ordered graphs avoiding $H$ 
for some ordered graphs $H$ with small  $f_\prec(H)$.  A specific example for such a graph $H$ is an ordered path $u_1u_2u_3u_4$, with $u_1\prec u_2\prec u_3\prec u_4$.
One can see from  Theorem~\ref{others} that $f_\prec(H)= 3$, but a complete bipartite ordered graph $G$ with all vertices of one bipartition class to the left of all other vertices does not contain $H$ and has $|V(G)|^2/4$ edges.  
However, for some ordered graphs $H$ with interval chromatic number $2$, one can show that ${\rm ex}_\prec(n,H)$ is linear. This in turn, implies that $f_\prec(H)$ is finite. 

Some of the extensive research on forbidden binary matrices and extremal functions for ordered graphs can be found in~\cite{CyclicOrderedGraphs,FuerediHajnal, Klazar_DavenportSchinzelSeq, Klazar_SmallGraphs, MarcusTardos}.

There are also connections between the Ramsey numbers $R_\prec(H)$ for ordered graphs and the function $f_\prec(H)$.
If the edges of $K_n$, $n= R_{\prec} (H) -1$,  are colored in two colors without monochromatic copies of $H$, then both color classes form ordered graphs $G_1$ and $G_2$ not containing $H$ as an ordered subgraph.
Then one of the $G_i$'s has chromatic number at least $\sqrt{n}$, since a product of proper colorings of $G_1$ and $G_2$ yields a proper coloring of $K_n$.
Therefore $f_\prec(H) \geq \sqrt{R_\prec(H)-1}$.
Ordered Ramsey numbers were recently studied by Conlon \textit{et al.}~\cite{ConlonFoxLeeSudakov} and Balko \textit{et al.}~\cite{BalkoCibulkaKralKyncl}.
Other research on ordered graphs includes characterizations of classes of graphs by forbidden ordered subgraphs~\cite{Damaschke, Ginn} and  the study of perfectly ordered graphs~\cite{Chvatal}.
  

\section{Structural Lemmas and Reductions}\label{reductions}

In this section we first analyze the structure of ordered trees without bonnets and tangled paths.
This  leads to a proof of Theorem~\ref{structural} in Section~\ref{proofs}.
Afterwards we establish several cases  when  $f_\prec(H)$ can be upper bounded in terms of $f_\prec(H')$ for a subgraph $H'$ of $H$.
This allows us to reduce the problem of whether $f_\prec(H) \neq \infty$ to the problem of whether $f_\prec(H') \neq \infty$.
These \emph{reductions} are the crucial tools in the proof of Theorem~\ref{non-crossing} in Section~\ref{proofs}.

 
\begin{lemma}\label{lem:edgeCoversVertex}
 Let $T$ be an ordered tree that does not contain a tangled path and let $u\prec v\prec w$ be vertices in $T$.
 If $uw$ is an edge in $T$, then all vertices of the path connecting $u$ and $v$ in $T$ are between $u$ and $w$.
\end{lemma}
\begin{proof}
 Let $P$ be the path in $T$ that starts with $v$ and ends with the edge $uw$.
 Let $\ell$ denote the leftmost vertex in $P$.
 Assume for the sake of contradiction that $\ell\prec u$.
 Then the path $vP\ell$ contains neither $u$ nor $w$ and therefore crosses the edge $uw$.
 Hence the paths $P\ell$ and $\ell P$ cross and $P$ is tangled, a contradiction.
 Therefore $\ell = u$.
 Due to symmetric arguments $w$ is the rightmost vertex in $P$.
 Hence all vertices in $P$ are between $u$ and $w$.
\end{proof}

\begin{lemma}\label{lem:noNewInnerCut}
 Let $T$ be an ordered tree  that contains neither a bonnet nor a tangled path and that has  only one segment .
 Deleting any leaf from $T$ yields an ordered tree that contains neither a bonnet nor a tangled path and that has  only one segment.
\end{lemma}
\begin{proof}
 Let $uv$ be an edge in $T$ incident to a leaf $u$ and let $T'=T-u$.
 Then clearly $T'$ is an ordered tree that contains neither a bonnet nor a tangled path.
 For the sake of contradiction assume that  $T'$ has at least two segments and let $x$ be an inner cut vertex in $T'$.
 Then $x\neq u$,$v$ and is between $u$ and $v$ in $T$, since $x$ is not an inner cut vertex in $T$.
 By reversing $T$ if necessary we may assume that $v\prec x\prec u$.
 Let $P$ be the $v$-$x$-path in $T'$.
 All vertices in $P$ are between $v$ and $u$ by Lemma~\ref{lem:edgeCoversVertex} applied to $u$, $v$ and $x$.
 In addition no vertex in $P$ is to the right of $x$ since $x$ is an inner cut vertex in $T'$.
 So all vertices in $P$ are between $v$ and $x$.
 Let $vw$ denote the first edge of $P$ and let $xy$ denote an edge in $T'$ with $x\prec y$.
 Such an edge $xy$ exists since the inner cut vertex $x$ is not rightmost in $T'$ and $T'$ is connected.
 If $u\prec y$, then $uvPxy$ is a tangled path in $T$.
 If $y\prec u$, then $u$, $v$, $w$, $x$ and $y$ form a bonnet in $T$.
 In both cases we have a contradiction and hence $T'$ has only one segment.
\end{proof}

\begin{lemma}\label{lem:goodTreeBipartite}
 If $T$ is an ordered tree that contains neither a bonnet nor a tangled path and that has only one segment, then $\chi_\prec(T)\leq 2$.
\end{lemma}
\begin{proof}
 We prove the claim by induction on $k=|V(T)|$.
 If $k\leq 2$, then clearly $\chi_\prec(T)\leq 2$.
 So assume that $k\geq 3$.
 Let $u$ denote a leaf in $T$, $v$ its neighbor in $T$, and let $T'=T-u$.
 Then $T'$ has only one segment and contains neither a bonnet nor a tangled path due to Lemma~\ref{lem:noNewInnerCut}.
 Inductively $\chi_\prec(T')\leq 2$, i.e., there is a partition $L\dot\cup R=V(T')$, with $L \prec R$, such that all edges in $T'$ are between $L$ and $R$.
 By reversing $T$ if necessary we assume that $v\in L$.
 For the sake of contradiction assume that $\chi_\prec(T)>2$.
 Then $u\prec \ell$ for the rightmost vertex $\ell$ in $L$, possibly $\ell = v$.
 Let $w\in R$ denote one fixed neighbor of $v$ in $T'$.
 Then all vertices of the path connecting $\ell$ and $v$ in $T'$ are between $v$ and $w$ due to Lemma~\ref{lem:edgeCoversVertex}.
 In particular $\ell$ is incident to an edge $\ell x$, $x\in R$, with $x\preceq w$.
 Hence $u\prec v$, since otherwise there is a bonnet on vertices $v,u,\ell,x$, and $w$ in $T$.
 If there is a vertex $y$, $u\prec y\prec v$, then all vertices of the path connecting $y$ and $u$ in $T$ are between $u$ and $v$ due to Lemma~\ref{lem:edgeCoversVertex}.
 But this is not possible since $y,v\in L$ and all the neighbors of $y$ are in $R$.
 Hence $u$ is immediately to the left of $v$ in $T$.
 Note that $u$ is not leftmost in $T$, since otherwise $v$ is an inner cut vertex in $T$.
 Consider the path $P$ connecting a vertex left of $u$ to $\ell$ in $T$.
 This path contains distinct vertices $p,q\in L$, $r\in R$, such that $pr$ and $rq$ are edges in $P$ and $p\prec u\prec v\preceq q\prec r$.
 Hence there is a bonnet, a contradiction.
 This shows that $\chi_\prec(T)\leq 2$.
\end{proof}


We now present several reductions.
Let us mention that some of the following arguments are similar to reductions used for extremal numbers of matrices~\cite{PachTardos, Tardos}.

Recall, that an inner cut vertex $v$ of an ordered graph $H$ splits $H$ into ordered graphs $H_1$ and $H_2$, where $H_1$ is induced by all vertices $u$ with $u\preceq v$ in $H$ and $H_2$ is induced by all vertices  $u$ with $v\preceq u$.
See Figure~\ref{fig:Reductions} (left).
\begin{figure}
\begin{minipage}[b]{0.28\textwidth}
 \centering
 \includegraphics{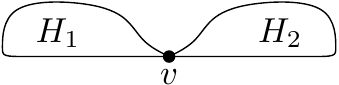}
\end{minipage}
 \hfill
 \begin{minipage}[b]{0.3\textwidth}
 \centering
 \includegraphics{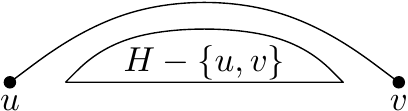}
\end{minipage}
\hfill
\begin{minipage}[b]{0.38\textwidth}
 \centering
 \includegraphics{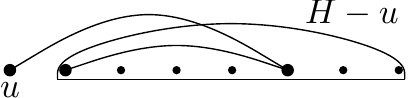}
\end{minipage}
\caption{An inner cut vertex $v$ splitting an ordered graph into ordered graphs $H_1$ and $H_2$ (left), an isolated edge $uv$ in an ordered graph $H$  (middle), and a reducible vertex $u$ (right).}
\label{fig:Reductions}
\end{figure}

\begin{reduction}\label{red:innerCut}
 If an inner cut vertex $v$ splits an ordered graph $H$ into ordered graphs $H_1$ and $H_2$ with $f_\prec(H_1),f_\prec(H_2) \neq \infty$, then
 \[
  f_\prec(H) \leq f_\prec(H_1) + f_\prec(H_2).
 \]
\end{reduction}
\begin{proof}
 Consider an ordered graph $G\in{\rm Forb}_\prec(H)$.
 Let $V_1$ denote the set of vertices in $G$ that are rightmost in some copy of $H_1$ in $G$.
 Further let $V_2=V(G)\setminus V_1$.
 Then $G[V_2]\in {\rm Forb}_\prec(H_1)$ by the choice of $V_1$.
 Moreover $G[V_1]\in {\rm Forb}_\prec(H_2)$, since otherwise the leftmost vertex $u$ in a copy of $H_2$ in $G[V_1]$ is also a rightmost vertex in a copy of $H_1$ and hence plays the role of $v$ in a copy of $H$ in $G$.
 Thus $\chi(G) \leq \chi(G[V_1]) + \chi(G[V_2]) \leq f_\prec(H_2) + f_\prec(H_1)$ and since $G \in{\rm Forb}_\prec(H)$ was arbitrary we have $f_\prec(H) \leq f_\prec(H_1) + f_\prec(H_2)$.
\end{proof}


\begin{reduction}\label{red:isolatedVertex}
 If $v$ is an isolated vertex in an ordered graph $H$ with $|V(H)|\geq 3$ and $f_\prec(H-v) \neq \infty$, then $f_\prec(H) \leq 2\,f_\prec(H-v).$ 
\end{reduction}
\begin{proof}
 Consider an ordered graph $G\in{\rm Forb}_\prec(H)$. 
If $v$ is not leftmost or rightmost in $H$, then let $V_1$ denote a set of every other vertex in $G$ and let $V_2=V(G)\setminus V_1$.
 Then $G[V_1], G[V_2]\in {\rm Forb}_\prec(H-v)$, since for any two vertices $u\prec w$ in $V_i$ there is a vertex $v\in V_{3-i}$ with $u\prec v\prec w$, $i=1,2$.
 Hence $\chi(G) \leq \chi(G[V_1])+\chi(G[V_2]) \leq 2f_\prec(H-v)$.
 If $v$ is the leftmost or the rightmost in $H$, assume without loss of generality the former.
 Then clearly $G-u \in {\rm Forb}_\prec(H-v)$ for the leftmost vertex $u$ of $G$. Thus $\chi(G) \leq 1+ \chi(G-u) \leq 1+ f_{\prec}(H-v) \leq 2f_{\prec}(H-v)$.
 Since $G \in{\rm Forb}_\prec(H)$ was arbitrary we have $f_\prec(H) \leq 2f_\prec(H-v)$ in both cases.
\end{proof}


\begin{reduction}\label{red:isolatedEdge}
 Let $u$ and $v$ be the leftmost and rightmost vertices in an ordered graph $H$, $|V(H)|\geq 4$.
 If $uv$ is an isolated edge in $H$ and $f_\prec(H-\{u,v\}) \neq \infty$, then
 \[
  f_\prec(H) \leq 2\,f_\prec(H-\{u,v\})+1.
 \] 
\end{reduction}
\begin{proof}
 See Figure~\ref{fig:Reductions} (middle).
 Let $H'=H-\{u,v\}$ and consider an ordered graph $G\in{\rm Forb}_\prec(H)$.
 If $G$ does not contain a copy of $H'$, then $\chi(G) \leq f_\prec(H') \leq 2f_\prec(H')+1$. So, assume that $G$ contains a copy of $H'$.
 Let $V_1\dot\cup\cdots\dot\cup V_p$ denote a partition of $V(G)$ into disjoint intervals with $V_1 \prec \cdots \prec V_p$, $v_i$ being the leftmost vertex in $V_i$, $1\leq i\leq p$, such that $G[V_i]\in{\rm Forb}_\prec(H')$, $1\leq i\leq p$, and $G[V_i\cup \{v_{i+1}\}]$ contains a copy of $H'$, $1\leq i<p$.
 Note that one can find such a partition greedily by iteratively choosing a largest interval from the left that does not induce any copy of $H'$ in $G$.
 If $p\geq 3$, there are no edges $xy$ with $x\in V_i$ and $v_{i+2}\prec y$, since otherwise $xy$ together with a copy of $H'$ in $G[V_{i+1}\cup\{v_{i+2}\}]$ forms a copy of $H$, $1\leq i\leq p-2$.
 
 Choose a set $\Phi$ of $2\,f_\prec(H')+1$ distinct colors.
 Let $\Phi_1,\ldots,\Phi_p\subset \Phi$ denote subsets of colors such that $|\Phi_i|=f_\prec(H')$, $1\leq i\leq p$, $\Phi_i\cap\Phi_{i+1}=\emptyset$, $1\leq i<p$, and, if $p\geq 3$, $\Phi_{i+2}\setminus (\Phi_i\cup\Phi_{i+1})\neq \emptyset$, $1\leq i\leq p-2$.
 Note that such sets $\Phi_i$ can be chosen greedily from $\Phi$.
 Since $G[V_i]\in{\rm Forb}_\prec(H')$ we can color $G[V_i]$ properly with colors from $\Phi_i$, $1\leq i\leq p$, such that, if $i\geq 3$, $v_i$ is colored with a color in $\Phi_i\setminus(\Phi_{i-1}\cup\Phi_{i-2})$.
 This yields a proper coloring of $G$ using colors from the set $\Phi$ only.
 Hence $\chi(G) \leq 2\,f_\prec(H')+1$.
 Since $G \in{\rm Forb}_\prec(H)$ was arbitrary we have $f_\prec(H) \leq 2\,f_\prec(H-\{u,v\})+1$.
\end{proof}

Recall, that a vertex in an ordered graph $H$ is called \emph{reducible}, if it is a leaf in $H$, is leftmost or rightmost in $H$ and has a common neighbor with the vertex next to it.
See Figure~\ref{fig:Reductions} (right).


\begin{reduction}\label{red:outerReduction}
 Let $H$ denote an ordered graph with $|V(H)|\geq 3$.
 If $u$ is a reducible vertex in $H$ and $f_\prec(H-u) \neq \infty$, then
 \[
  f_\prec(H) \leq 2\, f_\prec(H-u).
 \]
 Moreover, for each $G\in{\rm Forb}_\prec(H)$ there is $G'\subseteq G$ such that $G'$ is $1$-degenerate and deleting the edges of $G'$ from $G$ yields a graph from ${\rm Forb}_\prec(H-u)$. 
\end{reduction}
\begin{proof}
 By reversing $H$ if necessary we may assume that the reducible vertex $u$ is leftmost in $H$.
 Let $G\in{\rm Forb}_\prec(H)$.
 Let $E$ denote the set of edges in $G$ consisting for each vertex $w$ in $G$ of the longest edge to the left incident to $w$ in $G$, if such an edge exists.

 Assume that there is a copy $H'$ of $H-u$ in $G-E$.
 Let $v$ denote the vertex in $H'$ corresponding to the vertex immediately to the right of $u$ in $H$ and let $w$ denote the vertex in $H'$ corresponding to the neighbor of $u$ in $H$.
 Then $v$ is leftmost in $H'$ and there is an edge between $v$ and $w$ in $H'$.
 Thus, there is an edge $xw$ in $E$ incident to $w$ in $G$ with $x\prec v$.
 Hence $H'$ extends to a copy of $H$ in $G$ with the edge $xw$, a contradiction.
 This shows that $G-E\in{\rm Forb}_\prec(H-u)$.
 
 Finally observe that the graph $G'$ with the edge-set  $E$ is $1$-degenerate and hence $2$-colorable.
 This shows that $\chi(G) \leq \chi(G')\chi(G-E) \leq 2f_\prec(H-u)$ and since $G \in{\rm Forb}_\prec(H)$ was arbitrary we have $f_\prec(H) \leq 2f_\prec(H-u)$.
\end{proof}


Having Reduction Lemma~\ref{red:outerReduction} at hand, we are now ready to prove that every non-crossing monotonically alternating tree $T$ satisfies $f_\prec(T) \neq \infty$.

\begin{lemma}\label{lem:ReducMonAlt}
 If $T$ is a non-crossing monotonically alternating tree with $|V(T)|\geq 2$, then
 \[
  f_\prec(T) \leq 2|V(T)|-3.
 \]
\end{lemma}
\begin{proof}
 Let $k = |V(T)|$ and $G\in{\rm Forb}_\prec(T)$.
  We shall prove that $G$ can be edge-decomposed into $(k-2)$ $1$-degenerate graphs by induction on $k$.
  
  If $k=2$, then $T$ consists of a single edge only. Hence $G$ has an empty edge-set and there is nothing to prove.

So consider $k\geq 3$ and  assume that the induction statement holds for all smaller values of $k$.
 Assume for the sake  of contradiction that the leftmost vertex $u$ and the rightmost $w$ in $T$ are of degree at least $2$.
 Then the longest and the shortest edge  incident to $w$ do not coincide. Let  $e$ be the longest edge incident to $w$.
 Since in a monotonically alternating tree each edge is the shortest edge incident to its left or right endpoint,  $e$ is the shortest edge incident to its left endpoint.
 In particular, $e\neq uw$ because $u$ is incident to another edge $e'$, shorter than $uw$.  Thus $e$ and $e'$  cross since $\chi_\prec(T)\leq 2$, a contradiction.
 Hence the leftmost or the rightmost vertex is a leaf in $T$.
 
 By reversing $T$ if necessary we assume that $u$ is of degree $1$. We shall show that $u$ is a reducible leaf.
 To do so, we need to show that the vertex 
$x$ that is  immediately to the right of $u$ is  adjacent to the neighbor $v$ of $u$.  
Assume for the sake of contradiction that $x$ is not adjacent to $v$. 
Note that $v$ is adjacent to a leaf, so it is  not a leaf itself. 
Let $e''$ be an edge incident to $v$, $e''\neq uv$. Then 
an edge incident to $x$ crosses either $uv$ or $e''$ since $\chi_\prec(T)\leq 2$, a contradiction.
Thus  $x$ is adjacent to $v$ and $u$ is a reducible leaf in $T$.

 Therefore, by Reduction Lemma~\ref{red:outerReduction}, there is a $1$-degenerate subgraph $G'$ of $G$ such that removing the edges of $G'$ from $G$ yields a graph $G''\in{\rm Forb}_\prec(T-u)$.
 Observe that the tree $T-u$  is non-crossing and monotonically alternating with $k > |V(T-u)|=k-1 \geq 2$.
 Hence $G''$ can be edge-decomposed into $(k-3)$ $1$-degenerate graphs $G_1,\ldots,G_{k-3}$ by induction.
 Thus the graphs $G_1,\ldots,G_{k-3}, G'$ decompose $G$ into $(k-2)$ $1$-degenerate graphs, proving the induction step.

  If $k=2$, we know that $G$ has no edges and $\chi(G) =1 \leq 2|V(T)|-3$. 
So assume that $k\geq 3$.  Singe   $G$ is a union of $(k-2)$ $1$-degenerate graphs,  each subgraph of $G$ is a union of $(k-2)$ $1$-degenerate graphs, so each subgraph $G^*$  of $G$ on at least one vertex  that has at most $(k-2)(|V(G^*)|-1)$ edges, and thus has a vertex of degree at most $2(k-2)-1$. Therefore  $G$ is $(2(k-2)-1)$-degenerate, so 
 $\chi(G)\leq 2(k-2) \leq 2|V(T)|-3$.  Since $G \in{\rm Forb}_\prec(H)$ was arbitrary we have $f_\prec(H) \leq 2|V(T)|-3$.
\end{proof}

\begin{reduction}\label{red:MatchingReduction}
  Let $T$ denote an ordered matching on at least $2$ edges.
  If $uv$ is an edge in $T$ and $u$ and $v$ are consecutive and $f_\prec(T-\{u,v\}) \neq \infty$, then
  \[
   f_\prec(T) \leq 3\, f_\prec(T-\{u,v\}).
  \]
\end{reduction}
\begin{proof}
 Let $G\in{\rm Forb}_\prec(T)$ with vertices $v_1\prec\cdots\prec v_n$.
 We shall prove that $\chi(G) \leq 3\, f_\prec(T-\{u,v\})$ by induction on $n=|V(G)|$.
 If $n\leq 3\, f_\prec(T-\{u,v\})$, then the claim holds trivially.
 So assume that $n > 3\, f_\prec(T-\{u,v\}) \geq 3$.
 If there are two consecutive vertices $x$, $y$ in $G$ that are not adjacent, then let $G'$ denote the graph obtained by identifying $x$ and $y$.
 Then $G'\in{\rm Forb}_\prec(T)$ and $\chi(G) \leq \chi(G')$.
 Hence $\chi(G) \leq \chi(G') \leq 3\, f_\prec(T-\{u,v\})$ by induction.
 If each pair of consecutive vertices in $G$ forms an edge, then consider a partition $V(G)=V_0\dot\cup V_1\dot\cup V_2$ such that  $V_i=\{v_j\in V(G)\mid j\equiv i\pmod{3}\}$.
 Observe that for each pair of vertices $x,y\in V_i$ there are at least two adjacent vertices from $V(G)\setminus V_i$ between $x$ and $y$.
 Hence $G[V_i]\in{\rm Forb}_\prec(T-\{u,v\})$, $i=0,1,2$, since any copy of $T-\{u,v\}$ in $G[V_i]$ extends to a copy of $T$ in $G$.
 Hence $\chi(G) \leq 3\, f_\prec(T-\{u,v\})$ and since $G \in{\rm Forb}_\prec(H)$ was arbitrary we have $f_\prec(H) \leq 3\, f_\prec(T-\{u,v\})$.
\end{proof}



\section{Proofs of Theorems}\label{proofs}


\subsection{Proof of Theorem~\ref{Tutte-Shift}}

We will prove that if an ordered graph $H$ contains a cycle, a tangled path or a bonnet then for each positive integer $k$ there is an ordered graph $G\in{\rm Forb}_\prec(H)$ with $\chi(G) \geq k$.
 
\bigskip
 
First assume that $H$ contain a cycle of length $\ell$.
Fix a positive integer $k$ and  consider a graph $G$ of girth at least $\ell+1$ and chromatic number at least $k$ that exists by~\cite{LargeGirthHighChromatic}.
Then no ordering of the vertices of $G$ gives  an ordered subgraph isomorphic to $H$.
This shows that for any positive integer $k$, $f_\prec(H)\geq k$ and hence $f_\prec(H)=\infty$.

\bigskip
 
A tangled path is minimal if it does not contain a  proper subpath that is tangled.
Next we shall show that for each minimal tangled path $P$ and each $k\geq 1$ there is an ordered graph $G_k\in {\rm Forb}_\prec(P)$ with $\chi(G_k)\geq k$.

 By reversing $P$ if necessary we assume that in $P$ the paths $Pu$ and $uP$ cross for the rightmost vertex $u$ in $P$.
 We will prove the claim by induction on $k$.   If $k\leq 3$ let $G_k=K_k$ that has no crossing edges and thus no tangled paths.
  Consider $k\geq 4$ and let $G_{k-1}$ denote an $n$-vertex  graph of chromatic number at least $k-1$  that  does not contain a copy of $P$.
  Such a graph exists by induction. 
 The following construction is due to Tutte (alias Blanche Descartes) for unordered graphs~\cite{TutteConstruction}.
 Let  $N=(k-1)(n-1)+1$ and $M= \binom{N}{n} $.
 Consider pairwise disjoint sets  of vertices $U_1, \ldots, U_M,  V$ such that $|U_i|= n$, $i = 1, \ldots, M$,  $|V| = N$ and $U_1\prec\cdots\prec U_M\prec V$.
Let $V_1,\ldots,V_M$ be the $n$-element subsets of $V$. Let each  $U_i$, $i=1, \ldots, M$,  induce a copy of $G_{k-1}$.
Finally let  there be a perfect matching between $U_i$ and $V_i$ such that the $j^{\text{th}}$ vertex in $U_i$ is matched to the $j^{\text{th}}$ vertex in $V_i$,  $i=1, \ldots,  M$.
 See Figure~\ref{fig:Tutte}.

 \begin{figure}
 \centering
 \includegraphics{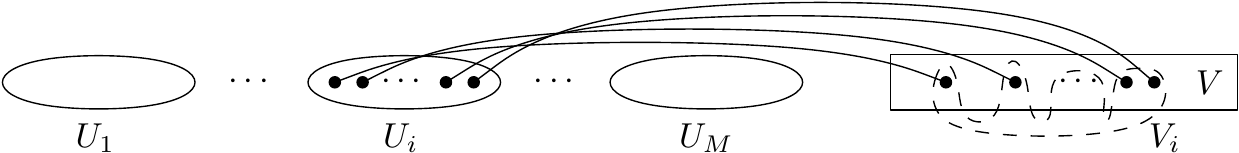}
 \caption{A graph $G_k$ obtained by Tutte's construction from a graph $G_{k-1}$. Here $G_k[U_i]=G_{k-1}$, $1\leq i\leq M$.}
 \label{fig:Tutte}
 \end{figure}
 
 \medskip
 
 First we shall show that $\chi(G_k) \geq  k$. 
 If there are at most $k-1$ colors assigned to  the vertices of $G_k$, then by Pigeonhole Principle there  are $n$ vertices of $V$ of the same color, i.e., there  
 is a set $V_i$ with all vertices of the same color,  say color $1$.
 Since each vertex of $U_i$ is adjacent to a vertex in $V_i$,  no vertex in $U_i$ is colored $1$, so if the coloring is proper, then $G[U_i]$ uses at most $k-2$ colors. 
 Hence the coloring is not proper, since $\chi(G[U_i])=\chi(G_{k-1})\geq k-1$.
 Therefore $\chi(G_k) \geq k$.
 
 Now, we shall show that $G_k$ does not contain a copy of $P$.
 Assume that there is such a copy $P'$ of $P$  in $G_k$ with rightmost vertex $u$ of $P'$.
 Let $x$ and $y$ be the neighbors of $u$ in $P'$, i.e., $P'$ is a union of paths $P'yu$ and $uxP'$.
 Then $u\in V$ and $x,y\not\in V$, since  $G[U_i]$ does not contain a copy of $P$ and there are no edges in $G_k[V]$.
 Let $x\in U_i$ and $y\in U_j$. Note that $i\neq j$ because the edges between $U_i$ and $V$ form a matching.
 The path $uxP'$ is a proper subpath of $P'$ and hence is not tangled.
 Recall that for each edge $zw$ with $z\in U_i$, $w\in V$, and $w\prec u$, we have $z\prec x$ due to the construction of the matching between $U_i$ and $V_i$.
 Hence the path $uxP'$ does not contain any vertex $w\in V$ with $w\prec u$, since otherwise the path $uxP'w$ has a vertex left of $x$ contradicting Lemma~\ref{lem:edgeCoversVertex} applied to $u$, $x$ and $w$.
 Hence $V(xP')\subseteq U_i$, because there are no edges between $U_i$'s and $u$ is rightmost in $P'$.
 See Figure~\ref{fig:NoTuttePath}.
 Similarly, all vertices of $P'y$ are contained in $U_j$.
 Thus $P'u$ and $uP'$ do not cross. However, $P'$ is a copy of $P$ with respective subpaths crossing, a contradiction. 
 Hence $G_k\in {\rm Forb}_\prec(P)$.

Now, if an ordered graph $H$ contains a tangled path, then it contains a minimal tangled path.
Thus $f_\prec(H)=\infty$.
\begin{figure}
\centering
\includegraphics{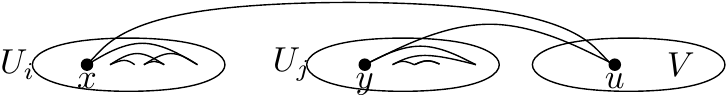}
\caption{A path in $G_k$ with rightmost vertex $u\in V$ is not tangled if $Pu$ and $uP$ are not tangled.}
\label{fig:NoTuttePath}
\end{figure}

\bigskip

Now, let $B$ be a bonnet.
By reversing $B$ if necessary, we assume that $B$ has vertices  $u\prec v  \preceq x,y\preceq w$ and edges $uv$, $uw$, $xy$.
A shift graph $S(n)$ is defined on vertices $\{(i,j)\mid 1\leq i<j\leq n\}$ and edges $\{\{(i,j),(j,t)\}\mid 1\leq i<j<t\leq n\}$.
We will show that some ordering of $S(n)$ does not contain $B$.
Let $G=S(n)$ be a shift graph  with vertices ordered  lexicographically, i.e., $(x_1,x_2)\prec (y_1,y_2)$ if and only if $x_1<y_1$, or $x_1=y_1$ and $x_2<y_2$.  
Assume that $G$ contains  vertices $u =(u_1,u_2)$, $v=(v_1, v_2)$, $x=(x_1, x_2)$, $y=(y_1, y_2)$ and $w=(w_1, w_2)$ that form  a copy of $B$ with $u\prec v  \preceq x,y\preceq w$ and edges $uv$, $uw$, $xy$.
Then $u_2=v_1$,  $u_2=w_1$, $x_2= y_1$. Thus $v_1=w_1$. However, since $v\preceq x,y\preceq w$,  we have that $v_1\leq x_1, y_1 \leq w_1$, so $x_1=y_1=v_1=w_1$. 
But $x_2=y_1$, thus $x_2=x_1$, a contradiction.
Thus $G\in{\rm Forb}_\prec(B)$.
We claim that $\chi(G) \geq \log(n) \geq \log c|V(G)|$.
Indeed consider a proper coloring $\phi$ of $G$ using $\chi(G)$ colors and sets of colors $\Phi_i=\{\phi(i,j)\mid i < j\leq n\}$, $1\leq i\leq n$.
Then $\phi(i,j)\not\in \Phi_j$, since a vertex $(i,j)$ is adjacent to all vertices $(j,t)$, $j<t\leq n$.
Therefore $\Phi_i\neq \Phi_j$ for all $j<i$.
Hence all the sets of colors are distinct.
This shows that $2^{\chi(G)}\geq n$, since there are at most $2^{\chi(G)}$ distinct subsets of colors.
This proves that $\chi(G) \geq \log(n)$.
Thus, for any $k$, there is an ordered graph of chromatic number at least $k$ in ${\rm Forb}_\prec(B)$.
So, if an ordered graph $H$ contains a bonnet, then $f_\prec(H)=\infty$.
\qed


\subsection{Proof of Theorem~\ref{structural}}


 Let $T'$ be a segment of an ordered tree that does not contain a bonnet or a tangled path.
 We shall prove that $T'$ is monotonically alternating by induction on $k=|V(T')|$. 
 Every ordered tree on at most two vertices is monotonically alternating.
 So suppose $k\geq 3$.
 We have $\chi_\prec(T') = 2$ due to Lemma~\ref{lem:goodTreeBipartite}.
  
 \begin{claim}
  The leftmost or the rightmost vertex in $T'$ is of degree $1$.
 \end{claim}
 \begin{claimproof}
  For the sake of contradiction assume that both the leftmost vertex $u$ and the rightmost vertex $v$ in $T'$ are of degree at least $2$.
  If $u$ and $v$ are adjacent then the edge $uv$, another edge incident to $u$ and another edge incident $v$ form a tangled path since $\chi_\prec(T')= 2$, a contradiction.
  If $u$ and $v$ are not adjacent let $P$ denote the path in $T'$ connecting $u$ and $v$.
  It uses at most one of the edges incident to $u$.
  Then any other edge $zu$ incident to $u$ crosses the edge in $P$ that is incident to $v$ since $\chi_\prec(T')= 2$.
  Hence $zP$ forms a tangled path, a contradiction.
  This shows that at least one of $u$ or $v$ is a leaf in $T'$.
 \end{claimproof}

 By reversing $T'$ if necessary we assume that the leftmost vertex $u$ is a leaf in $T'$.
 The ordered tree $T'-u$ is monotonically alternating by induction and Lemma~\ref{lem:noNewInnerCut}. 
Consider  the partition $V(T')=L\dot\cup R$, with $L\prec R$ and $L$ and $R$ being independent sets. Such a partition  is unique since $T'$ is connected.
 Let $v$ be the neighbor of $u$ in $T'$.  Since $\chi_\prec(T')=2$, $v\in R$.
 Since $T'$ is connected, $k\geq 3$ and $u$ is leftmost in $T'$, the edge $uv$ is not the shortest edge incident to $v$.
 Hence $uv\not\in S(R)$ and therefore $S(R)$ has no crossing edges by induction.
 Clearly $uv\in S(L)$ since $uv$ is the only edge incident to $u$ and thus it is the shortest incident to $u$ edge.
 If $uv$ crosses some edge $xy$ in $T'$, $x\prec y$, then all vertices in the path connecting $v$ and $x$ are between $x$ and $v$ due to Lemma~\ref{lem:edgeCoversVertex} 
 applied to $x$, $y$ and $v$.
 Therefore $xy$ is not the shortest edge incident to $x$ and hence $xy\not\in S(L)$.
 This shows that $S(L)$ has no crossing edges and thus $T'$ is monotonically alternating.
 
 The other way round assume that each segment of an ordered tree $T$ is monotonically alternating.
 We need to show that each segment contains neither a bonnet nor a tangled path.
 Let $T'$ denote a segment of $T$, $V(T') = L\cup R$, $L\prec R$ and $E(T') = S(L) \cup S(R)$, so each edges is either a shortest  edge incident to a vertex in $R$ or 
 a shortest edge incident to a vertex in $L$.
 Then $\chi_\prec(T')\leq 2$ and hence $T'$ does not contain a bonnet.
 We will prove that $T'$ does not contain a tangled path by induction on $k=|V(T')|$.
 If $k\leq 3$, then there are no crossing edges in $T'$ and hence no tangled path.
 Suppose $k\geq 4$.

 Assume that the leftmost vertex $u$ and the rightmost vertex $w$ in $T'$ are of degree at least $2$.
 If $uw\in E(T')$ then $uw \not \in S(L)$ and $uw\not\in S(R)$, a contradiction.  So, $uw\not \in E(T')$. 
 Consider  the longest edge $xw$ incident to $w$. Then $x\neq u$ and since $xw\not \in S(R)$, $xw\in S(L)$.
 Then the shortest edge incident to $u$ crosses $xw$, a contradiction since $S(L)$ does not contain crossing edges.
 Hence the leftmost or the rightmost vertex is a leaf in $T'$.
 
 By reversing $T'$ if necessary we assume that the leftmost vertex $u$ is a leaf.
 We see that $T'-u$ is monotonically alternating, thus by induction it does not contain a tangled path.
 Hence if $T'$ has a tangled path $P$, then $P$ contains an edge $uv$ crossing some other edge in $P$, where $v$ is the neighbor of $u$ in $T'$. 
 Then the rightmost vertex $r$ in $P$ is of degree $2$ and to the right of $v$, since $P$ is tangled and $u$ is leftmost and of degree $1$ in $T'$.
 Let $x$ and $y$, $x\prec y$, be neighbors of $r$ in $P$.
 Then $xr$ is the shortest edge incident to $x$, since any shorter edge forms a tangled path with $r$ and $y$ in $T'-u$.
 This is a contradiction since $uv$ and $xr$ cross and $T'$ is monotonically alternating.
 Thus $T'$ has no tangled path.
 
 \bigskip
 
 Finally we prove the last statement of the theorem.
 If $H$ is a connected ordered graph with $f_\prec(H)\neq\infty$, then $H$ is a tree that contains neither a bonnet nor a tangled path due to Theorem~\ref{Tutte-Shift}.
 Hence each segment of $H$ is a monotonically alternating tree.
 \qed


\subsection{Proof of Theorem~\ref{non-crossing}}


Let $T$ be  a non-crossing ordered graph such that $f_\prec(T) \neq \infty$.
Then $T$ is acyclic, contains no tangled path and no bonnet by Theorem~\ref{Tutte-Shift}.
Hence $T$ is a non-crossing ordered forest with no bonnet. \\

On the other hand let $T$ be a non-crossing forest with no bonnet.
Recall that $f_\prec(H)\geq k-1$ for each ordered $k$-vertex graph $H$ because $K_{k-1}\in{\rm Forb}_\prec(H)$.
We shall prove that $f_\prec(T)\neq\infty$.
Let $k = |V(T)|$ and consider any ordered graph $G\in {\rm Forb}_\prec(T)$.
We will prove by induction on $k$ that $\chi(G)\leq 2^k$  and $\chi(G)\leq 2k-3$ if $T$ is a tree. 
 If $k = 2$, then clearly $\chi(G) = 1$.  So consider $k\geq 3$.\\

If $T$ is a tree, then each segment of $T$ is a monotonically alternating tree, by Theorem~\ref{structural}.
If there is only one segment in $T$, then $f_\prec(T) \leq 2k-3$ by Lemma~\ref{lem:ReducMonAlt}. 
If there is more than one segment in $T$, then there is an inner cut vertex splitting $T$ into two trees $T_1$ and $T_2$ that are clearly also non-crossing and contain no bonnet.
Thus by Reduction Lemma~\ref{red:innerCut} and induction we have  $f_\prec(T) \leq f_\prec(T_1) + f_\prec(T_2) \leq  2|V(T_1)| -3 + 2|V(T_2)| -3 = 2(|V(T)| +1) - 6 = 2k -4.$\\

If  $T$  is  a forest we consider several cases.
If $T$ has more than one segment, then there is an inner cut vertex splitting $T$ into two forests $T_1$ and $T_2$ that are clearly also non-crossing and contain no bonnet.
Thus by Reduction Lemma~\ref{red:innerCut}  and induction we have  $f_\prec (T) \leq f_\prec(T_1) + f_\prec(T_2) \leq  2^{|V(T_1)|} + 2^{|V(T_2)|} = 2^{t} + 2^{k+1-t} \leq 2^k$ with $t=|V(T_1)|\geq 2$. If $T$ has an isolated vertex $u$, then by Reduction Lemma~\ref{red:isolatedVertex}  and induction we have $f_\prec(T) \leq 2 f_\prec(T-u) \leq 2\cdot2^{k-1} = 2^k$.
Finally, if $T$ has no isolated vertices and exactly one segment, then consider the leftmost and rightmost vertices $u$ and $v$ of $T$.  Since $u$ and $v$ are not isolated in this case, 
and $T$ is non-crossing with no inner cut vertices,  $uv$ is an edge.  If $uv$ is isolated, then $k\geq 4$ (since there is no isolated vertex) and by Reduction Lemma~\ref{red:isolatedEdge} and induction we have $f_\prec(T) \leq 2\cdot f_\prec(T-\{u,v\})+1 \leq 2\cdot2^{k-2}+1 \leq 2^k.$
 If $uv$ is not isolated, then either $u$ or $v$, say $u$, is a leaf of $T$, since $T$ is non-crossing and does not contain a bonnet.
Let $xv$ denote the longest edge incident to $v$ in $T- u$. Note that $x$ exists since the edge $uv$ is not isolated.
Then there is no other vertex between $u$ and $x$, since such a vertex would be isolated in the non-crossing forest $T$ without bonnets.
Thus, $u$ is a reducible vertex, so by Reduction Lemma~\ref{red:outerReduction} and induction we have $f_\prec(T) \leq 2 f_\prec(T-u) \leq 2\cdot 2^{k-1} = 2^k$.\\

Next, we provide a $k$-vertex non-crossing tree with no bonnet such that $\infty\neq f_\prec(T)\geq k$.
Let $T$ be a monotonically alternating path on $k\geq 4$ vertices with leftmost vertex of degree $1$, as in Figure~\ref{fig:MoserSpindle} (right).
Further let $G$ denote a graph on vertices $u\prec x_1\prec\cdots\prec x_{k-2}\prec y_1\prec\cdots\prec y_{k-2}\prec x\prec y$ such that $xy$ is an edge and $\{u,x_1,\ldots,x_{k-2}\}$, $\{u,y_1,\ldots,y_{k-2}\}$, $\{x,x_1,\ldots,x_{k-2}\}$, and $\{y,y_1,\ldots,y_{k-2}\}$ induce complete graphs on $k-1$ vertices each.
See Figure~\ref{fig:MoserSpindle} (left).

We shall show that $G\in{\rm Forb}_\prec(T)$ and $\chi(G)\geq k$.
Consider a proper vertex coloring of $G$ using colors $1,\ldots,k-1$.
Without loss of generality $u$ has color $1$.
Then all colors $2,\ldots,k-1$ are used on the vertices $x_1,\ldots, x_{k-2}$ as well as on $y_1,\ldots, y_{k-2}$.
Hence both $x$ and $y$ are of color $1$, a contradiction.
Thus $\chi(G)\geq k$.

Assume that there is a copy $P$ of $T$ in $G$.
Let $v$ be the leftmost and $w$ be the rightmost vertex in $P$.
Note that $vw$ is an edge and that there are $k$ vertices between $v$ and $w$.
Therefore $vw$ is one of the edges $uy_i$, $1\leq i\leq k-2$, $x_jx$, $1\leq j\leq k-2$, or $y_1y$.
In the first case $V(P)\subseteq \{u,y_1,\ldots,y_{k-2}\}$, in the second case $V(P)\subseteq \{x_1,\ldots,x_{k-2},x\}$ and in the last case either $P=y_1,y,x$ or $V(P)\subseteq \{y,y_1,\ldots,y_{k-2}\}$.
Since $T$ has at least $4$ vertices, $P\neq y_1,y,x$.
So in any case $P$ has at most $k-1$ vertices, a contradiction since $T$ has $k$ vertices.
Hence $G\in{\rm Forb}_\prec(T)$.\\
 
 \begin{figure}
  \centering
  \includegraphics{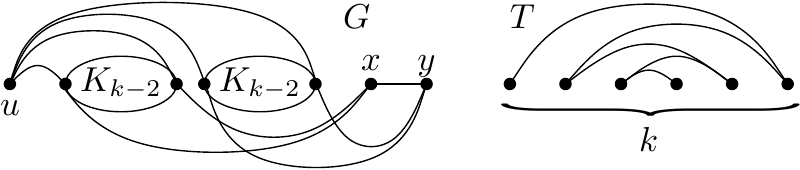}
  \caption{An ordered graph $G$ with chromatic number $k$ not containing a non-crossing and ordered tree $T$ on $k$ vertices without bonnets on the right, $k=6$.}
  \label{fig:MoserSpindle}
 \end{figure}

Finally it is easy to see that $f_\prec(T) = k-1$ for any ordered tree $T$ on at most $3$ vertices using Reduction Lemmas~\ref{red:innerCut} and~\ref{red:outerReduction}.\qed

 
\subsection{Proof of Theorem~\ref{others}}


\begin{itemize}
\item Let $T$ be an ordered forest on $k$ vertices where each segment is a generalized star, a $2$-nesting, or a $2$-crossing.
Let $T_1,\ldots,T_s$ denote the segments of $T$ and $k_i=|V(T_i)|$, $1\leq i\leq s$.
Let $T'$ be a segment of $T$.
If $T'$ is a generalized star on $k'$ vertices, then the center of the star is leftmost (or rightmost) in $T'$.
Let $G\in{\rm Forb}_\prec(T')$.
Then each vertex in $G$ has at most $k'-2$ neighbors to the right (or to the left).
Thus each such graph can be greedily colored from right to left (or left to right) with at most $k'-1$ colors.
This shows that $f_\prec(T') \leq |V(T')|-1$.
If $T'$ is a $2$-nesting, then $f_\prec(T')=3=|V(T')|-1$ due to~\cite{DW04} (Lemma 9).
If $T'$ is a $2$-crossing, then $f_\prec(T')=3=|V(T')|-1$, since any graph not containing $T'$ is outerplanar and outerplanar graphs have chromatic number at most $3$.
We apply Reduction Lemma~\ref{red:innerCut} and the results above which yield $f_\prec(T) \leq \sum_{i=1}^s f_\prec(T_i) \leq \sum_{i=1}^s (k_i-1) = k-1$.

\item Let $T$ be an ordered forest on $k$ vertices where each segment is a generalized star, a non-crossing tree without bonnets, a crossing or a nesting.
Let $T_1,\ldots,T_s$ denote the segments of $T$ and $k_i=|V(T_i)|\geq 2$.
Let $T'$ be a segment of $T$.
If $T'$ is a $k'$-nesting or a $k'$-crossing, $k'\geq 2$, then $f_\prec(T') \leq 4(k'-1)  \leq 2|V(T')|-3$ due to equation~(\ref{f-ex}), since any graph $G\in{\rm Forb}_\prec(T')$ contains less than $2(k'-1)|V(G)|$ edges due to Dujmovic and Wood~\cite{DW04} (for nestings), respectively Capoyleas and Pach~\cite{CapoyleasPach} (for crossings).
Further $f_\prec(T')\leq 2|V(T')|-3$ if $T'$ is a non-crossing tree without bonnets due to Theorem~\ref{non-crossing}.
Hence Reduction Lemma~\ref{red:innerCut} yields $f_\prec(T) \leq \sum_{i=1}^s f_\prec(T_i) \leq \sum_{i=1}^s (2k_i-3) \leq 2k-3$.

\item Let $T=M(t,m, \pi)$ for some positive integers $m$ and $t$ and a permutation $\pi$ of $[t]$.
If $t=1$, then $f_\prec(T) = m$ due to the results above, since $M(1,m, \pi)$ is a star on $m+1$ vertices.
Weidert~\cite{Weidert} proves that ${\rm ex}_\prec(n,M(t,1, \pi)) \leq {\rm ex}_\prec(n,M(t,2, \pi))\leq 11 t^4\binom{2t^2}{2t} n< t^4(2t^2)^{2t}n$ for any positive integer $t\geq 2$ and any permutation $\pi$ of $[t]$.
Moreover if $m\geq 2$, then $${\rm ex}_\prec(n,M(t,m, \pi))\leq 2^{t(m-2)}{\rm ex}_\prec(n,M(t,2, \pi))$$ due to a reduction by Tardos~\cite{Tardos}.  Therefore ${\rm ex}_\prec(n,M(t,m, \pi)) $ $< 2^{tm}t^{4+4t}n$.
Thus, using the fact that $|V(T)|=k=tm+t$ and equation (\ref{f-ex})  we have that $f_\prec(M(t,m, \pi)) \leq 2^{tm+9t\log(t)}\leq 2^{10 k\log k}$.

\item Conlon \textit{et al.}~\cite{ConlonFoxLeeSudakov} and independently Balko \textit{et al.}~\cite{BalkoCibulkaKralKyncl} prove that that there is a positive constant $c$ such that for any sufficiently large positive integer $k$ there is an ordered matchings on $k$ vertices with ordered Ramsey number at least $2^{c\frac{\log(k)^2}{\log\log(k)}}$.
If, for some ordered graph $H$, the edges of a complete ordered graph $G$ on $N=R_\prec(H)-1$ vertices are colored in two colors without monochromatic copies of $H$, then both color classes form ordered graphs $G_1$ and $G_2$ in ${\rm Forb}_\prec(H)$.
Then one of the $G_i$'s has chromatic number at least $\sqrt{N}$, since a product of proper colorings of $G_1$ and $G_2$ yields a proper coloring of $G$ using $\chi(G_1)\chi(G_2)\geq \chi(G)=N$ colors.
This shows that there is a positive constant $c'$ such that for all positive integers $k$ and ordered matchings $H$ on $k$ vertices with $f_\prec(H)\geq 2^{c'\frac{\log(k)^2}{\log\log(k)}}$.
\qed
\end{itemize}


 \section{Small Forests}\label{small-forests}

 \newcommand{\bast}{{\bf \textasteriskcentered}}
 \newcommand{\ord}{{\rm ord}}

Let $P_k$ denote a path on $k$ vertices, $M_k$ a matching on $k$ edges and $S_k$ a star with $k$ leaves (note that $M_1 = S_1 = P_2$ and $P_3=S_2$).
Further let $G+H$ denote the vertex disjoint union of graphs $G$ and $H$.
Then the set of all forests without isolated vertices and at most $3$ edges is given by
\[\{P_2, S_2, M_2, S_3, P_4, S_2+P_2, M_3\}.\]

Let $G$ denote a graph on $n$ vertices and $a$ automorphisms.
Then the number $\ord(G)$ of non-isomorphic orderings of $G$ equals $\ord(G) = \frac{n!}{a}$.
Hence 
\begin{align*}&\ord(P_2)=\tfrac{2!}{2}=1,\enskip \ord(S_2) =\tfrac{3!}{2}=3,\enskip \ord(M_2) =\tfrac{4!}{8}=3,\enskip \ord(S_3) =\tfrac{4!}{3!}=4,\\
&\ord(P_4) =\tfrac{4!}{2}=12,\enskip \ord(S_2+P_2) =\tfrac{5!}{2\cdot 2}=30,\enskip \ord(M_3) =\tfrac{6!}{6\cdot 4\cdot 2}=15.
\end{align*}

Recall that the reverse $\overline{T}$ of an ordered graph $T$ is the ordered graph obtained by reversing the ordering of the vertices in $T$.
Note that $f_\prec(T)=f_\prec(\overline{T})$ for any ordered graph $T$ since $G\in{\rm Forb}_\prec(T)$ if and only if $\overline{G}\in{\rm Forb}_\prec(\overline{T})$.
Table~\ref{fig:tableForests} shows all ordered forests $T$ without isolated vertices and at most $3$ edges and their $f_\prec$ values, where only one of $T$ and  $\overline{T}$ is listed.
So when $T$ and $\overline{T}$ are not isomorphic ordered graphs the entry in the table represents two graphs.
Such cases are marked with an~\bast.
For example there are only two instead of three entries for $S_2$ and similarly for the other graphs.

\begin{figure}
\begin{tabular}{r|c|cc|ccc|}
 {\bf T} & \includegraphics[scale=1]{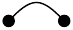} & \includegraphics[scale=1]{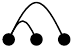} & \includegraphics[scale=1]{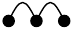} & \includegraphics[scale=1]{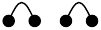} & \includegraphics[scale=1]{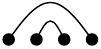} & \includegraphics[scale=1]{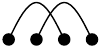}\\[1ex]
 $\mathbf{f_\prec(T)}$ & $1$ & $2$ \bast & $2$ & $3$ & $3$ & $3$\\
& (Thm. \ref{others}) & (Thm. \ref{others}) & (Thm. \ref{others}) & (Thm. \ref{others}) & (Thm. \ref{others}) & (Thm. \ref{others})
\end{tabular}
\vspace{1em}

\begin{tabular}{r|cc|ccc}
 {\bf T} & \includegraphics[scale=1]{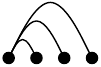} & \includegraphics[scale=1]{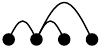} & \includegraphics[scale=1]{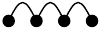} & \includegraphics[scale=1]{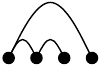} & \includegraphics[scale=1]{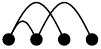} \\[1ex]
 $\mathbf{f_\prec(T)}$ & $3$ \bast & $3$ \bast & $3$ & $\infty$ \bast & $\infty$ \bast \\
& (Thm. \ref{others}) & (Thm. \ref{others}) & (Thm. \ref{others}) & (bonnet) & (tangled)
\end{tabular}

\vspace{1em}

\begin{tabular}{r|ccccc|}
 {\bf T} & \includegraphics[scale=1]{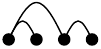} & \includegraphics[scale=1]{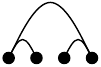} & \includegraphics[scale=1]{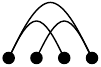} & \includegraphics[scale=1]{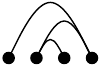} & \includegraphics[scale=1]{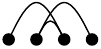} \\[1ex]
 $\mathbf{f_\prec(T)}$ & $3$ \bast & $\infty$ & $\infty$ & $4$ \bast & $\leq 4$ \\
& (Thm. \ref{others}) & (bonnet) & (tangled) & (Lem. \ref{lem:ReducMonAlt}, Fig.~\ref{fig:MoserSpindle}) & (Red. \ref{red:outerReduction})
\end{tabular}

\vspace{1em}

\begin{tabular}{r|ccccc}
 {\bf T} & \includegraphics[scale=1]{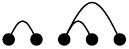} & \includegraphics[scale=1]{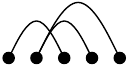} & \includegraphics[scale=1]{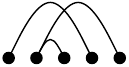} & \includegraphics[scale=1]{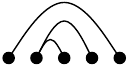} & \includegraphics[scale=1]{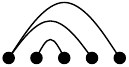} \\[1ex]
 $\mathbf{f_\prec(T)}$ & $4$ \bast & $\leq 6$ \bast & ? \bast & $\leq 6$ \bast & $\leq 6$ \bast \\
& (Thm. \ref{others}) &(Red. \ref{red:outerReduction}) &  & (Red. \ref{red:isolatedEdge}) & (Lem. \ref{lem:ReducMonAlt})
\end{tabular}

\vspace{1em}

\begin{tabular}{r|ccccc}
 {\bf T} & \includegraphics[scale=1]{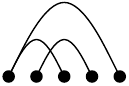} & \includegraphics[scale=1]{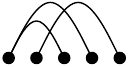} & \includegraphics[scale=1]{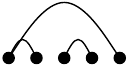} & \includegraphics[scale=1]{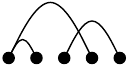} & \includegraphics[scale=1]{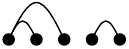} \\[1ex]
 $\mathbf{f_\prec(T)}$ & ? \bast & $\neq\infty$ \bast & $\infty$ \bast & ? \bast & $4$ \bast \\
&  & (Thm. \ref{others}) & (bonnet) &   & (Thm. \ref{others})
\end{tabular}

\vspace{1em}

\begin{tabular}{r|cccccc|}
 {\bf T} & \includegraphics[scale=1]{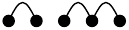} & \includegraphics[scale=1]{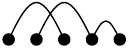} & \includegraphics[scale=1]{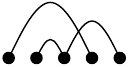} & \includegraphics[scale=1]{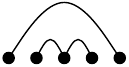} & \includegraphics[scale=1]{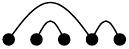} & \includegraphics[scale=1]{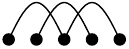}\\[1ex]
 $\mathbf{f_\prec(T)}$ & $4$ \bast & $4$ \bast & ? \bast & $\leq 6$ & $4$ \bast & ?\\
& (Thm. \ref{others})  & (Thm. \ref{others}) &  & (Red. \ref{red:isolatedEdge})  & (Thm. \ref{others}) &
\end{tabular}

\vspace{1em}

\begin{tabular}{r|ccccc}
 {\bf T} & \includegraphics[scale=1]{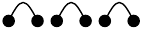} & \includegraphics[scale=1]{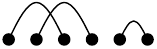} & \includegraphics[scale=1]{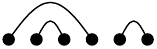} & \includegraphics[scale=1]{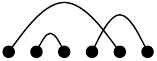} & \includegraphics[scale=1]{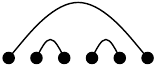} \\[1ex]
 $\mathbf{f_\prec(T)}$ & $5$ & $5$ \bast & $5$ \bast & $\leq 9$ \bast & $\leq 7$ \\
& (Thm. \ref{others})  &  (Thm. \ref{others}) & (Thm. \ref{others}) & (Red. \ref{red:MatchingReduction})  & (Red. \ref{red:isolatedEdge})
\end{tabular}

\vspace{1em}

\begin{tabular}{r|cccccc}
 {\bf T} & \includegraphics[scale=1]{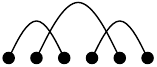} & \includegraphics[scale=1]{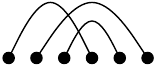} & \includegraphics[scale=1]{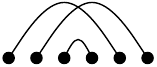} & \includegraphics[scale=1]{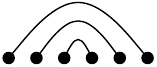} & \includegraphics[scale=1]{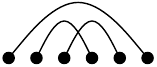} & \includegraphics[scale=1]{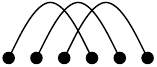}\\[1ex]
 $\mathbf{f_\prec(T)}$ & ? & $\neq\infty$ \bast & $\leq 9$ & $\leq 8$ & $\leq 7$ & $\leq 8$\\
&   & (Thm. \ref{others}) & (Red. \ref{red:MatchingReduction}) & (Thm. \ref{others})  & (Red. \ref{red:isolatedEdge}) & (Thm. \ref{others})
\end{tabular}

\vspace{1em}

\caption{All ordered forests $T$ on at most $3$ edges without isolated vertices and their~$f_\prec$~value.}
\label{fig:tableForests}
\end{figure}


 \section{Conclusions}\label{conclusions} 
In this paper, we consider the function $f_\prec(H) = \sup\{\chi(G)\mid G\in {\rm Forb}_\prec(H)\}$ for ordered graphs $H$ on at least $2$ vertices.
We prove that in contrast to unordered and directed graphs, $f_\prec(H)=\infty$ for some ordered forests $H$.
To this end we explicitly describe several infinite classes of minimal ordered forests $H$ with $f_\prec(H) = \infty$.
A full answer to the following question remains open.

\begin{question}\label{ques:generalCharacterization}
 For which ordered forests $H$ does $f_\prec(H) = \infty$ hold?
\end{question}

We completely answer Question~\ref{ques:generalCharacterization} for non-crossing ordered graphs $H$.
Suppose that $H$ is a non-crossing ordered $k$-vertex graph with $f_\prec(H)\neq\infty$.
We prove that, if $H$ connected, then $k-1\leq f_\prec(H)\leq 2k-3$ and, if $H$ is disconnected, then $k-1\leq f_\prec(H)\leq 2^k$.
In addition, we give infinite classes of graphs for which $f_\prec(H)= |V(H)|-1$, as well as infinite classes of graphs for which $|V(H)|\leq f_\prec(H) \neq \infty$.
Note that we do not know whether $f_\prec(H)\neq \infty$ for the matchings in the last statement of Theorem~\ref{others}.
For crossing connected ordered graphs, we reduce Question~\ref{ques:generalCharacterization} to monotonically alternating trees:

\begin{question}\label{ques:monAlt}
For which monotonically alternating trees $H$ does $f_\prec(H) = \infty$ hold?
\end{question}

We do not have an answer to Question~\ref{ques:monAlt} even for some monotonically alternating paths.
A smallest unknown such path is $u_5u_1u_3u_2u_4$, where $u_1\prec \cdots \prec u_5$.
See Figure~\ref{fig:unknown} (left).
The situation becomes even more unclear for crossing disconnected graphs.
We do not know the value of $f_\prec(H)$ for some ordered matchings $H$.
A smallest such matching has edges $u_1u_3$, $u_2u_5$ and $u_4u_6$ where $u_1\prec\ldots\prec u_6$.
See Figure~\ref{fig:unknown} (right).
\begin{figure}
 \begin{minipage}{0.48\textwidth}
 \centering
  \includegraphics{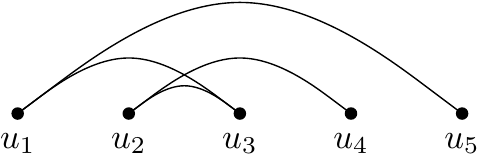}
 \end{minipage}
\hfill
\begin{minipage}{0.48\textwidth}
\centering
  \includegraphics{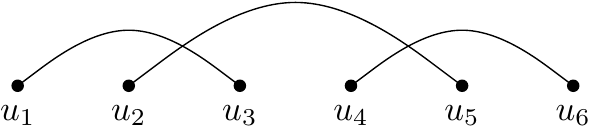}
 \end{minipage}
 \caption{Ordered graphs $H$ for which we don't know whether $f_\prec(H)=\infty$.}
 \label{fig:unknown}
\end{figure}
Note that Reduction Lemmas~\ref{red:innerCut},~\ref{red:isolatedVertex},~\ref{red:isolatedEdge} and~\ref{red:outerReduction} apply to crossing ordered graph as well.
We find a more precise version of  Reduction Lemma~\ref{red:isolatedVertex} and other types of reductions, similar to reductions for matrices in~\cite{Tardos}, but none of these lead to significantly better upper bounds in Theorems~\ref{non-crossing} and~\ref{others} or a new class of forests with finite $f_\prec$. The following question remains open, even when restricted to non-crossing graphs.
 
\begin{question} For $k\geq 4$, what is the value of the function
\[f_\prec(k) = \max \{ f_\prec(H)\mid ~~  |V(H)|=k,~  f_\prec(H)\neq \infty\}?\]
\end{question}

\bibliographystyle{abbrv}
\bibliography{lit}
\end{document}